\documentclass[12pt]{amsart}
 \usepackage{amscd}
 \usepackage{amssymb}
 \usepackage{hyperref}
 \usepackage[all]{xypic}
 \usepackage{url}
 \usepackage{graphicx}
 \usepackage{amsmath}

 \def\max{\operatorname{max}}

 \newtheorem{lemma}{Lemma}[section]
 \newtheorem{corollary}[lemma]{Corollary}
 \newtheorem{theorem}[lemma]{Theorem}

 \newtheorem{example}[lemma]{Example}

 \advance\headheight1.15pt

\begin{document}

\baselineskip=20pt
\title[GENERALIZATIONS OF SOME CONCENTRATION INEQUALITIES]{GENERALIZATIONS OF SOME CONCENTRATION INEQUALITIES}

\author[M. Ashraf Bhat]{M. Ashraf Bhat}

\address{Department of Mathematics, Indian Institute of Technology Ropar, Punjab-140001, India.}
 {\email{ashraf74267@gmail.com}

\author{G. Sankara Raju Kosuru}
\address{Department of Mathematics, Indian Institute of Technology Ropar, Punjab-140001, India.}
\email{raju@iitrpr.ac.in}



 \maketitle
 \begin{abstract}
For a real-valued measurable function $f$ and a nonnegative, nondecreasing function $\phi$, we first obtain a Chebyshev type inequality which provides an upper bound for
 $\displaystyle \phi(\lambda_{1}) \mu(\{x \in \Omega : f(x) \geq \lambda_{1} \}) + \sum_{k=2}^{n}\left(\phi(\lambda_{k})- \phi(\lambda_{k-1})\right) \mu(\{x \in \Omega : f(x) \geq \lambda_{k}\}) ,$ where $0  < \lambda_1 < \lambda_2 \cdots \lambda_n < \infty$. Using this, generalizations of a few concentration inequalities such as
 Markov, reverse Markov, Bienaym\'e-Chebyshev, Cantelli and Hoeffding inequalities are obtained.
 \end{abstract}
 
 %%%%%%%%%%%%%%%%%%%%%%%%%%%%%%%%%%%%%%%%%%%%%%%%%%%

 \section{Introduction}

The Chebyshev inequality (Measure-theoretic version) states (\cite{Royden2010}) that for any extended
real-valued measurable function $f$ on a measure space $(\Omega, \Sigma, \mu)$
and $\lambda >0,$
\begin{equation}\label{Chebyshev1}
\mu(\{x \in \Omega : |f(x)| \geq \lambda \}) \leq \frac{1}{\lambda} \int_{\Omega} |f| d \mu .
\end{equation}
%This is sometimes also called as Markov's inequality. 
In general, if $f$ is a real-valued measurable function and $\phi$ is a nonnegative and nondecreasing extended
real-valued Lebesgue measurable function on $\mathbb{R}$  %or R# ?
 with $\phi(\lambda) \neq 0$, then we have 
\begin{equation} \label{Chebyshev2}
\mu(\{x \in \Omega : f(x) \geq \lambda \})=\mu(\{x \in \Omega : \phi \circ f(x) \geq \phi( \lambda) \})\leq \frac{1}{\phi(\lambda)} \int_{\Omega} \phi \circ f d \mu .
\end{equation}
If the measure space  is a probability space (with the probability $\mathbb{P}= \mu$ and the event space $\mathcal{F}= \Sigma)$
and $X$ is a real-valued random variable with expectation $\mathbb{E}(X)$, inequality $(\ref{Chebyshev1})$ turns out to be the Markov inequality, $\mathbb{P}(|X| \geq \lambda) \leq \frac{\mathbb{E}(|X|)}{\lambda}, ~\mbox{ for } \lambda > 0.$
Further, if $\mathbb{E}(X)< \infty$, we have the  Bienaym\'e-Chebyshev inequality
\begin{equation} \label{Chebyshev4}
\mathbb{P}(|X-\mathbb{E}(X)| \geq \lambda) \leq \frac{\sigma^2}{\lambda^2}, \mbox{ for } \lambda > 0,
\end{equation}
where $\sigma ^2 $ is the variance of the random variable $X$. 
Moreover, if $X \leq M$ for some $M >0$, then as special case of 
Markov's inequality, one can obtain the following reverse Markov's inequality 
(\cite{KumarSneha2019}):
\begin{equation}\label{reverseMarkov}
\displaystyle \mathbb{P}(X \leq \lambda) \leq \frac{M-\mathbb{E}(X)}{M-\lambda}, ~\forall ~ 0 < \lambda < M.
\end{equation}
Such type of inequalities, which provide bounds on deviation of a random variable from a given value (usually its expectation) 
are known as  concentration inequalities ({\cite{SB2013,RossProbability}). In addition to aforementioned inequalities, a few other concentration inequalities which are of our interest are Cantelli's inequality and Hoeffding's inequality. Cantelli's inequality (also known as the one-sided Chebyshev inequality) offers a better bound than (\ref{Chebyshev4}) in case of a single tail,
$\mathbb{P}(X-\mathbb{E}(X) \geq \lambda) \leq  \frac{\sigma^2}{\sigma^2+ \lambda^2}, ~\mbox{ for } \lambda > 0.$ 
Hoeffding's inequality (\cite{Hoeffding1963})
provides an upper bound for the probability of deviation of a sum of independent random variables  from its expected value. It 
states that if 
$X_i$ is a  bounded random variable in $[a_i,b_i]$ for $1 \leq i \leq n$, $X_i$'s are independent 
and $X=X_1+X_2+...+X_n$, then $\mathbb{P}(X-\mathbb{E}(X) \geq \lambda) \leq e^{\frac{-2\lambda^2}{\sum_{i=i}^{n}(b_i-a_i)^2}}, ~ \mbox{for} ~ \lambda >0.$

In 1998, James C. Owings posed a problem (\cite{ChapmanNester}), that lead to an extension of the Markov inequality. ``If $X$ is a random variable with expectation $\mathbb{E}(X)$ and variance $\sigma^2$. What is the smallest value of $c$ for which $\mathbb{P}(|X-\mathbb{E}(X)| \geq 2\sigma)+\mathbb{P}(|X-\mathbb{E}(X)| \geq 3 \sigma) \leq c  ?"$
Even though one may have $ \displaystyle \mathbb{P}(|X-\mathbb{E}(X)| \geq 2\sigma ) \leq \frac{1}{4}$, which can be observed by (\ref{Chebyshev4}),
in \cite{ChapmanNester}, the authors independently obtained that $c=\frac{1}{4}$  satisfies the inequality.
 Moreover, Chapman and Nester (\cite{ChapmanNester}) obtained that for any increasing sequence $t_1, t_2,...,t_n$ of positive numbers with $t_1 > 1$ the least upper bound of $\displaystyle \sum_{i=1}^{n}\mathbb{P}(X \geq t_i)$ is $\mbox{max}\displaystyle \left\{ \frac{i}{t_i} : 1 \leq i \leq n\right\} $. The same was further generalized by Eisenberg and Ghosh (\cite{eisenberg2001}).
It is worthwhile to mention that one might not find such analogous generalizations in the case of Cantelli's and Hoeffding's inequalities. For a few  other types of generalizations of Markov's  and Chebyshev's inequalities,  reader can refer to \cite{Bundy2014,Ghosh2002,Godwin1955,Mark2019,Leser1942,Mallows1956,Oklin1960,Haruhiko2020,Oklin1958}.
Some generalizations of Cantelli's and Hoeffding's inequalities can be found in \cite{Chen2011,Steven2011,Krafft1969,Peter2006,Haruhiko2019,Haruhiko2021}.

In this paper we commence with a  Chebyshev type inequality which 
provides an upper bound for the sums of the form,
$$\phi(\lambda_{1}) \mu(\{x \in \Omega : f(x) \geq \lambda_{1} \}) + \sum_{k=2}^{n}(\phi(\lambda_{k})- \phi(\lambda_{k-1})) \mu(\{x \in \Omega : f(x) \geq \lambda_{k}\}) ,$$
where $\phi$ is  a nonnegative and nondecreasing extended real-valued measurable function and $0<\lambda_{1}< \lambda_{2} < ... < \lambda_{n} < \infty $.
%a measure-theoretic form of Chebyshev's inequality to the sums
We use this inequality to extend the aforementioned concentration inequalities,  Markov's, Bienaym\'{e}-Chebyshev, Cantelli's and Hoeffding's. 
The techniques used here are very elementary and the results obtained are 
with their original premise.
To the best of our knowledge, these generalizations are not available in literature. 
Our results include few existing theorems.

%%%%%%%%%%%%%%%%%%%%%%%%%%%%%%%%%%%%%%%%%%%%%%%%%%%%%

 \section{A  generalization of Measure-theoretic Chebyshev's inequality} \label{sec1}

Let $f$ be a real-valued measurable function on a measure space $(\Omega,\Sigma, \mu)$. The distribution function, $\mu_f:[0, \infty) \mapsto [0,\infty]$, of $f$ is defined as $\mu_{f}(\lambda)=\mu(\{x \in \Omega : |f(x)| > \lambda\})$ for $ \lambda \geq 0$ (\cite{BennetSharpley}). The following lemma is useful in the sequel.

 \begin{lemma} \cite{BennetSharpley} \label{Lemma2.1}
Let $f$ be a real-valued measurable function on a measure space $(\Omega,\Sigma, \mu)$. Then
$$\int_{\Omega}|f| d \mu = \int_{0}^{\infty} \mu_{f}(\lambda) d \lambda .$$
 \end{lemma}

Now we prove a generalization of the Chebyshev inequality (\ref{Chebyshev2})
 \begin{theorem} \label{maintheorem}
Let $f$ be as above and $ 0 <  \lambda_{1} < \lambda_{2} < \cdots < \lambda_{n} < \infty$.
 Suppose $\phi$ is a nonnegative and nondecreasing extended real-valued measurable function on $\mathbb{R}$ with $\phi(\lambda_1) >0$.  Then
 \begin{equation*} \label{eq7}
%% \phi(\lambda_{1}) \mu(\{x \in \Omega : f(x) \geq \lambda_{1} \}) + 
% \sum_{k=2}^{n}(\phi(\lambda_{k})- \phi(\lambda_{k-1})) \mu(\{x \in \Omega : 
% f(x) \geq \lambda_{k}\}) \leq \int_{\Omega} \phi \circ f d\mu 
\displaystyle \sum_{k=1}^{n}\left[ \phi(\lambda_{k})- \phi(\lambda_{k-1})\right] \mu(\{x \in \Omega :  f(x) \geq \lambda_{k}\}) \leq \int_{\Omega} \phi \circ f d\mu,
 \end{equation*}
where $\phi(\lambda_0)=0$.
\end{theorem}
 \begin{proof}
 If $\mu(\{x \in \Omega : (\phi \circ f)(x)= \infty\})>0$, then the inequality trivially holds.
Also, if $\phi(\lambda_i)= \infty$ for some $i$, then $\mu(\{x \in \Omega : f(x) \geq \lambda_{j}\})  \leq \mu(\{x \in \Omega : (\phi \circ f)(x)= \infty \})=0,$ 
for all $i \leq j \leq n$.
Hence $ \displaystyle \sum_{k=1}^{n}\left[ \phi(\lambda_{k})- \phi(\lambda_{k-1})\right] \mu(\{x \in \Omega :  f(x) \geq \lambda_{k}\}) =\sum_{k=1}^{i-1}\left[\phi(\lambda_{k})- \phi(\lambda_{k-1})\right] \mu(\{x \in \Omega :  f(x) \geq \lambda_{k}\})$.
Therefore without loss of generality we may assume that $\mu(\{x \in \Omega : (\phi \circ f)(x)= \infty\})=0 $ and $\phi(\lambda_n) < \infty$. 
By Archimedean property there is a natural number $N_0$ with $$0 < \phi(\lambda_1)-\frac{1}{m} \leq \phi(\lambda_2)- \frac{1}{m} \leq ... \leq \phi(\lambda_n)- \frac{1}{m},~ \forall ~m \geq N_0.$$
 Since $\phi$ and $f$ are measurable functions and $\phi$ is nonnegative, we have that $\phi \circ f$ is a nonnegative measurable function. Define a simple function $s:[0,\infty) \to \mathbb{R}$ by
 $$s(\lambda) = \sum_{k=1}^{n} \alpha_k \scalebox{1.5}{$\chi$}_{I_k}(\lambda),~\mbox{ for } \lambda \in  [0,\infty),$$
 where $\displaystyle \alpha_k = \mu_{\phi \circ f}\left(\phi(\lambda_k)-\frac{1}{m}\right),~I_1=\left[0, \phi(\lambda_1)-\frac{1}{m}\right),~\displaystyle I_k = \left[\phi(\lambda_{k-1})-\frac{1}{m}, \phi(\lambda_k)-\frac{1}{m}\right)$ for $2 \leq k \leq n$ and $\scalebox{1.5}{$\chi$}_{I_k}$ is the characteristic function of $I_k ~(1 \leq k \leq n)$. By definition we have $s \leq \mu_{\phi \circ f}$. It follows that
 $$\int_{0}^{\infty}s(\lambda)d\lambda \leq \int_{0}^{\infty}\mu_{\phi \circ f}(\lambda)d\lambda.$$
Thus
\small 
 \begin{eqnarray*}
\displaystyle \left[\phi(\lambda_1)-\frac{1}{m}\right]\mu_{\phi \circ f}\left(\phi(\lambda_1)-\frac{1}{m} \right)& + & \sum_{k=2}^n \left[\left(\phi(\lambda_k)-\frac{1}{m}\right)-\left(\phi(\lambda_{k-1})-\frac{1}{m}\right)\right]\mu_{\phi \circ f}\left(\phi(\lambda_k)-\frac{1}{m}\right) \\ 
& = &\int_{0}^{\infty}s(\lambda)d\lambda \\
& \leq & \int_0^\infty \mu_{\phi \circ f} (\lambda)d \lambda\\
& = & \int_{\Omega} \phi \circ f d \mu, \quad ~~~~\mbox{     (by Lemma \ref{Lemma2.1})}.
 \end{eqnarray*}
 \normalsize
 Now, for all $1 \leq i \leq n$ we have, $$\displaystyle \{x \in \Omega : f(x) \geq \lambda_i\} \subseteq \{x \in \Omega : \phi \circ f (x) \geq \phi(\lambda_i)\} \subseteq  \left\{x \in \Omega : \phi \circ f(x) > \phi(\lambda_i)-\frac{1}{m}\right\}.$$
 Therefore by using monotonicity of measure and letting $m \to \infty$, we get
\small
  $$\phi(\lambda_{1}) \mu(\{x \in \Omega : f(x) \geq \lambda_{1} \}) + \sum_{k=2}^{n}(\phi(\lambda_{k})- \phi(\lambda_{k-1})) \mu(\{x \in \Omega : f(x) \geq \lambda_{k}\}) \leq \int_{\Omega} \phi \circ f d \mu .$$
 \normalsize
  Thus by setting $\phi(\lambda_0) =0$, we have
  $$\displaystyle \sum_{k=1}^{n}\left[ \phi(\lambda_{k})- \phi(\lambda_{k-1})\right] \mu(\{x \in \Omega :  f(x) \geq \lambda_{k}\}) \leq \int_{\Omega} \phi \circ f d\mu.$$ 
%  This completes the proof.
 \end{proof}
 
 %%%%%%%%%%%%%%%%%%%%%%%%%%%%%%%%%%%%%%%%%%%%%%%%%%%%%%%%%%%%%%%%%
 
  \section{Generalizations of concentration inequalities}
In this section we obtain generalizations of a few concentration inequalities. 
%We also provide a simple proof of strong law of large numbers which is based on the assumption that the variance of each random variable is finite.
We first prove the following generalization of Markov's inequality. 

\begin{theorem}[Generalization of Markov's inequality] \label{mg}
Let $ X$ be a real-valued random variable on a probability space $(\Omega,\mathcal{F}, \mathbb{P})$. Suppose $\phi$ is a nonnegative, nondecreasing  Lebesgue-measurable function on $\mathbb{R}$. If $0 < \lambda_{1} < \lambda_2 < \cdots < \lambda_{n} <\lambda_{n+1}= \infty$, then
\begin{equation*} \label{eqmarkov}
\displaystyle
%\sum_{k=1}^{n}\left[\phi(\lambda_{k})-\phi(\lambda_{k-1})\right]\mathbb{P}(X \geq \lambda_{k}) = 
\sum_{k=1}^{n} \phi(\lambda_{k}) \mathbb{P}(\lambda_{k} \leq X < \lambda_{k+1}) \leq  \mathbb{E}(\phi(X)).
\end{equation*}
\end{theorem}

\begin{proof}
By setting $\phi(\lambda_0)=0$ and using Theorem \ref{maintheorem}, we have
\begin{eqnarray*}
\displaystyle \sum_{k=1}^{n} \phi(\lambda_{k}) \mathbb{P}(\lambda_{k} \leq X < \lambda_{k+1}) & = &
\sum_{k=1}^{n}\left[\phi(\lambda_{k})-\phi(\lambda_{k-1})\right]\mathbb{P}(X \geq \lambda_{k}) \\
& \leq &  \int_{\Omega} \phi \circ X(\omega) d\mathbb{P}(\omega) =  \mathbb{E}(\phi(X)).
\end{eqnarray*} 
\end{proof}

Now, suppose $a_{1},a_{2},...,a_{n}$ are real numbers with $\underset{{1 \leq k \leq n}}{\max} a_{k} >0$ and $\underset{1 \leq k \leq n}{\max}~\displaystyle \frac{a_{k}}{\lambda_{k}} = \frac{a_{v}}{\lambda_{v}}$. By applying Theorem \ref{mg} for the function 
$$ \phi(x)= \begin{cases} x ,~~  x \geq 0 \\
 0, ~~\mbox{ otherwise} 
 \end{cases} $$
and the random variable $|X|$, we obtain 
\begin{equation} \label{eqngm}
\displaystyle \sum_{k=1}^{n} \lambda_{k}\mathbb{P}(\lambda_{k} \leq |X| < \lambda_{k+1}) \leq  \mathbb{E}(|X|).
\end{equation}
Multiplying both sides of the above inequality  by $\displaystyle \frac{a_{v}}{\lambda_{v}}$, we get the following corollary which is the main result in \cite{eisenberg2001}.
 
 \begin{corollary} \cite{eisenberg2001}
 Let $X,~\lambda_i's$ and $a_i's$ be as above. If $X$ is nonnegative, then 
 \begin{equation*}\label{EisenbergIneq}
\displaystyle  \sum_{i=1}^{n} a_i \mathbb{P}(\lambda_i \leq X < \lambda_{i+1}) \leq \frac{\mathbb{E}(X) a_v}{\lambda_v}
\end{equation*}
where $\underset{1 \leq k \leq n}{\max}~\displaystyle \frac{a_{k}}{\lambda_{k}} = \frac{a_{v}}{\lambda_{v}}$.
 \end{corollary}

Now we use Theorem \ref{mg} to extend the inequality (\ref{reverseMarkov}).
 \begin{corollary} [Generalization of reverse Markov's inequality] 
Let $X,\phi$ be as in Theorem \ref{mg}. Suppose $X \leq M$.
 Then for any sequence $0<\lambda_1 < \lambda_2< \cdots < \lambda_n < M$,
 $$\displaystyle \sum_{k=1}^{n}(\lambda_{k+1}-\lambda_{k})\mathbb{P}(X \leq \lambda_{k}) \leq M-  \mathbb{E}(X),$$
 where $\lambda_{n+1}=M$.
  \end{corollary}
 \begin{proof}
Set $\lambda_{k}^{'}=M-\lambda_{n-k+1}, \: k=1,2,\cdots,n$ and $\phi(x)=x$ if  $x\geq 0$ and $0$ otherwise.
By Theorem \ref{maintheorem}, we have 
$$\lambda_{1}^{'}\mathbb{P}(M-X \geq \lambda_{1}^{'}) + \sum_{k=2}^{n}(\lambda_{k}^{'}-\lambda_{k-1}^{'})\mathbb{P}(M-X \geq \lambda_{k}^{'}) \leq \mathbb{E}(M-X).$$
Thus
$$\mathbb{P}(X \leq \lambda_{n})+\sum_{k=2}^{n}\frac{(\lambda_{k}-\lambda_{k-1})}{M-\lambda_{n}}\mathbb{P}(X \leq \lambda_{k-1}) \leq\frac{ M-  \mathbb{E}(X)}{M-\lambda_{n}}.$$ This completes the proof.
 \end{proof}

Further applying Theorem $\ref{mg}$ to the random variable $|X-\mathbb{E}(X)|$ and for the function $\phi(x)=x^2$ if  $x\geq 0$ and $0$ otherwise, we get the following generalization of the Bienaym\'e-Chebyshev inequality.

 \begin{corollary}(Generalization of Chebyshev's inequality)\label{cg}.
 Let $X$ be a random variable on the probability space $\Omega$ with finite expected value $\mathbb{E}(X)$ and non zero variance $\sigma^2$. Then for any sequence  $ 0=\lambda_0 <  \lambda_{1} < \lambda_{2} < \cdots < \lambda_{n} < \infty$,
 \begin{equation*}
\sum_{k=1}^{n}(\lambda_{k}^2-\lambda_{k-1}^2)\mathbb{P}(|X-\mathbb{E}(X)| \geq \lambda_{k}) \leq  \sigma^2.
\end{equation*}
 \end{corollary}

The following theorem is a generalization of the Cantelli inequality.
 \begin{theorem} [Generalization of Cantelli's inequality] \label{CNG} Let X be a real-valued random variable with finite variance $\sigma^2$ and expected value $\mathbb{E}(X)$. Suppose  $ 0 <  \lambda_{1} < \lambda_{2} < \cdots < \lambda_{n} < \infty$. Then
 \begin{equation*}
 \displaystyle \sum_{k=1}^{n} \frac{(\lambda_{1}\lambda_{k}+\sigma^2)^2 -( \lambda_{1}\lambda_{k-1}+\sigma^2)^2}{(\lambda_{1}^2+\sigma^2)^2}\mathbb{P}(X-\mathbb{E}(X) \geq \lambda_{k}) \leq  \frac{\sigma^2}{\sigma^2+ \lambda_{1}^2},
%\leq  \frac{\sigma^2}{\sigma^2+ \lambda_{1}^2} .
% \mathbb{P}(X-\mathbb{E}(X) \geq \lambda_{1}) + \sum_{k=2}^{n} \frac{(\lambda_{1}\lambda_{k}+\sigma^2)^2 -( \lambda_{1}\lambda_{k-1}+\sigma^2)^2}{\lambda_{1}^2+\sigma^2}\mathbb{P}(X-\mathbb{E}(X) \geq \lambda_{k})
%\leq  \frac{\sigma^2}{\sigma^2+ \lambda_{1}^2} .
 \end{equation*}
 where $\lambda_0=\displaystyle - \frac{\sigma^2}{\lambda_1}$.
  \end{theorem}
 \begin{proof}
 By setting $Z=X-\mathbb{E}(X)$, we get $\mathbb{E}(Z)=0$ and variance of $Z$ is $\sigma^2$. Applying Theorem \ref{maintheorem}, for the random variable $Z$ and choosing the function  $\phi(x)(=\phi_t(x))=(x+t)^2$ when $x \geq 0$ and $0$ elsewhere (for a fixed $t\geq 0$), we have
$$\mathbb{P}(Z \geq \lambda_{1}) + \sum_{k=2}^{n} \frac{(\lambda_{k}+t)^2 -(\lambda_{k-1}+t)^2}{(\lambda_{1}+t)^2}\mathbb{P}(Z \geq \lambda_{k})
\leq  \frac{\mathbb{E}((Z+t)^2)}{(t+ \lambda_{1})^2}.$$
Thus 
\begin{equation} \label{eqnC3.5}
 \mathbb{P}(Z \geq \lambda_{1}) + \sum_{k=2}^{n} \frac{(\lambda_{k}+t)^2 -(\lambda_{k-1}+t)^2}{(\lambda_{1}+t)^2}\mathbb{P}(Z \geq \lambda_{k})
\leq  \frac{\sigma^2+t^2}{(t+ \lambda_{1})^2}.
\end{equation}
Since above holds for all $t \geq 0$, we choose the value of $t$ which minimizes the right hand side of (\ref{eqnC3.5}). For this, Let 
$$g(t)=\frac{\sigma^2 + t^2}{(\lambda_{1} + t)^2}.$$
One can obtain that $g(t)$ attains its minimum at $t_0=\frac{\sigma^2}{\lambda_{1}}$. By replacing $t$ by $t_0$ in (\ref{eqnC3.5}), 
$$ \mathbb{P}(X-\mathbb{E}(X) \geq \lambda_{1}) + \sum_{k=2}^{n} \frac{(\lambda_{1}\lambda_{k}+\sigma^2)^2 -( \lambda_{1}\lambda_{k-1}+\sigma^2)^2}{(\lambda_{1}^2+\sigma^2)^2}\mathbb{P}(X-\mathbb{E}(X) \geq \lambda_{k})
\leq  \frac{\sigma^2}{\sigma^2+ \lambda_{1}^2}.$$
This completes the proof.
 \end{proof}
Finnally, we prove a generalization of Hoeffding's inequality by using the following lemma.
\begin{lemma}[Hoeffding's Lemma]\cite{SB2013} \label{HL} If $X$ is a random variable with $\mathbb{E}(X)=0$ and $X \in [a,b]$, then for any $s>0$
$$\displaystyle \mathbb{E}\left(e^{sX}\right) \leq e^{\frac{s^2(b-a)^2}{8}}.$$
\end{lemma}

\begin{theorem} [Generalization of Hoeffding's inequality] Let $X_{1},X_{2},\cdots,X_{N}$ be independent random variables in $[a_i,b_i]$, $i\in \{1,2,\cdots ,N\}$ and $S_{N}=X_{1}+X_{2}+\cdots+X_{N}$.
Suppose  $ 0 <  \lambda_{1} < \lambda_{2} < \cdots < \lambda_{n} < \infty$. Then
%where $X_{i} \in [a_{i},b_{i}] \: \forall i$,
%Then
$$\displaystyle \sum_{k=1}^{n}\left[e^{(\lambda_{k}- \lambda_{1}){\frac{4\lambda_{1}}{\sum_{i=1}^{N}(b_i-a_i)^2}}}-e^{(\lambda_{k-1}-\lambda_{1}){\frac{4\lambda_{1}}{\sum_{i=1}^{N}(b_i-a_i)^2}}}\right]\mathbb{P}[S_{N}-\mathbb{E}(S_{N}) \geq \lambda_{k}] \leq e^{-\frac{2\lambda_{1}^2}{\sum_{i=1}^{N}(b_i-a_i)^2}},$$
where $\lambda_0 = -\infty.$ 
\end{theorem}
\begin{proof} By setting $\lambda_{n+1}=\infty$ and using (\ref{eqngm}) we have,
$$\displaystyle \sum_{k=1}^n \left[e^{(s\lambda_k)}\right]\mathbb{P}\left(e^{s\lambda_k} \leq e^{s\left(S_{N}-\mathbb{E}(S_{N})\right)} < e^{s\lambda_{k+1}}\right) \leq \mathbb{E}\left(e^{s\left(S_{N}-\mathbb{E}(S_{N})\right)}\right),~\mbox{ for all } s >0.$$
Thus, for any $s >0$,
\begin{eqnarray*}
 & & \displaystyle  e^{(s\lambda_{1})}\mathbb{P}[ (S_{N}-\mathbb{E}(S_{N})) \geq \lambda_{1} ] + \sum_{k=2}^{n}((e^{(s\lambda_{k})}-e^{(s\lambda_{k-1})})\mathbb{P}[ (S_{N}-\mathbb{E}(S_{N})) \geq \lambda_{k} ] \\
&=& e^{(s\lambda_{1})}\mathbb{P}[e^{s(S_{N}-\mathbb{E}(S_{N}))} \geq e^{s(\lambda_{1})}] + \sum_{k=2}^{n}((e^{(s\lambda_{k})}-e^{(s\lambda_{k-1})})\mathbb{P}[e^{s(S_{N}-\mathbb{E}(S_{N}))}  \geq e^{s(\lambda_{k})}] \\
&=& \sum_{k=1}^n e^{(s\lambda_k)}\mathbb{P}(e^{s\lambda_k} \leq e^{s(S_{N}-\mathbb{E}(S_{N}))} < e^{s\lambda_{k+1}}) \\
& \leq & \mathbb{E}(e^{s(S_{N}-\mathbb{E}(S_{N}))})\\
& \leq & \overset{N}{\underset{i=1}{\Pi}} \mathbb{E}(e^{s(X_i-\mathbb{E}(X_i))})\\
 & \leq & \overset{N}{\underset{i=1}{\Pi}} e^{\left(\frac{s^2(b_i-a_i)^2}{8}\right)},~~\mbox{ by Lemma }\ref{HL}. 
\end{eqnarray*}
Multiplying both sides of the above inequality by $e^{-s\lambda_1}$ and setting $\lambda_0=-\infty$, we get
\begin{equation} \label{eqnHI}
\displaystyle \sum_{k=1}^{n}\left[e^{(s(\lambda_{k}- \lambda{1})}-e^{s(\lambda_{k-1}-\lambda{1})}\right]\mathbb{P}[S_{N}-\mathbb{E}(S_{N}) \geq \lambda_{k}] \leq e^{\left(-s\lambda_1+\frac{s^2}{8}\sum_{i=1}^N(b_i-a_i)^2\right)}.
\end{equation}
Now we choose the value of $s$ which minimizes the right hand side of the above inequality. For this, let 
$$g(s)=\displaystyle -s\lambda_1+\frac{s^2}{8}\sum_{i=1}^N(b_i-a_i)^2, ~s>0.$$
One can obtain that $g(s)$ attains its minimum at $s_0=\frac{4\lambda_1}{\sum_{i=1}^N(b_i-a_i)^2}$.
The theorem follows by replacing $s$ with $s_0$ in (\ref{eqnHI}).
\end{proof}

\section{Discussion and Examples}

Using step function approximations, we obtained a measure theoretic generalization of the Chebyshev inequality. Using the same for a special case of the probability space, we obtained generalizations of the Markov, Chebyshev, Cantelli and Hoeffding inequalities. These results provide upper bounds for certain linear combinations of the probabilities of $n$ events, $n \geq 1$. It is not hard to see that for the case $n=1$, these inequalities boil down to their corresponding parent inequalities. 
Next we discuss some examples.

\begin{example} Suppose random variable $X$ represents the income of a population. Assume that $10 \%$ of the population have at least five times the average income. By applying Theorem \ref{maintheorem} with $n=2$, $\lambda_1=2\mathbb{E}(X)$, $\lambda_2= 5\mathbb{E}(X)$ and $\phi(x)=x$ for $x \geq 0$ and $0$ otherwise,
%$$ \phi(x)= \begin{cases} x ,~~  x \geq 0 \\
% 0, ~~\mbox{ otherwise} 
% \end{cases} ,$$
 we get
 \begin{equation*}\mathbb{E}(X)\mathbb{P}\left(X \geq 2 \mathbb{E}(X)\right)+ \left(5\mathbb{E}(X)-2 \mathbb{E}(X)\right) \mathbb{P}\left( X \geq 5 \mathbb{E}(X)\right) \leq \mathbb{E}(X).
 \end{equation*}
 This shows that, $\mathbb{P}\left(X \geq 2X \right) \leq \frac{7}{20} \times 100\% =35 \%.$ Thus, not more than $35\%$ of the population can have at least twice the average income.
\end{example}

\begin{example} Suppose that $10,000$ candidates appeared for a job interview with $550$ vacancies. The selection committee assigns some points to the candidates based on their performance in the interview.
% so that later they can select only $550$ top performers. 
Assume that the average score of the interview is $65$ with a standard deviation of $5$. Suppose exactly $100$ candidates have secured $90$ or more points. Can a candidate with 85 points get the job?  Let us apply Theorem \ref{CNG} with $n=2$, $\lambda_1 =20 $ and $\lambda_2=25$, we obtain
\begin{equation*}
\mathbb{P }(X-65 \geq 20)+ \frac{(20 \times 25 + 5^2)^2 - (20^2+ 5^2)^2}{(20^2+5^2)^2} \mathbb{P}(X-65 \geq 25) \leq \frac{5^2}{5^2+ 20^2}.
\end{equation*}
Using $\mathbb{P}(X \geq 90) = \frac{100}{10000}= \frac{1}{100}$ and simplifying, we get $\mathbb{P}(X \geq 85) \leq 0.0536.$
Thus, $0.0536 \times 10,000=536$ have secured $85$ or more points. Therefore with the available number of vacancies, he/she can be sure of the selection.
\end{example}

\begin{example}
As an application for determining confidence intervals, take $n=2$, $\lambda_1= \sigma$ and $ \lambda_2 = k \sigma$ in Corollary \ref{cg}, where $k >1$ and $\sigma >0$ is standard deviation of a random variable  $X$. We obtain,
{\scriptsize{
\begin{equation} \label{eqn1.4.1}
\mathbb{P}(|X - \mathbb{E}(X)| \geq k \sigma) \leq \frac{1}{k^2-1} \left[1- \mathbb{P}(|X - \mathbb{E}(X)| \geq  \sigma) \right]= \frac{1}{k^2-1}  \mathbb{P}(|X - \mathbb{E}(X)| < \sigma).
\end{equation}
}}
Thus, for any arbitrary distribution, if we know the probability of data within one standard deviation of the mean, we can calculate a lower bound for the probability of the data within $k$ standard deviations of the mean. For example, if for a distribution  $\mathbb{P}(|X - \mathbb{E}(X)| < \sigma) \leq 0.75$, then by equation (\ref{eqn1.4.1}) we have $\mathbb{P}(|X - \mathbb{E}(X)| < 3\sigma) \geq 1- \frac{1}{8} \times 0.75 \geq 0.90.$ 
\end{example}

%%%%%%%%%%%%%%%%%%%%%%%%%%%%%%%%%%%%%%%%%%%%%%%%%%%%%%%%%%%%%%%%%%%

\section*{Acknowledgements}
The authors would like to thank the reviewers for their technical comments and constructive suggestions to improve the manuscript. The authors are grateful to acknowledge Dr. Arun Kumar (IIT Ropar) for his valuable suggestions. Also, the first author (M. Ashraf Bhat) would like to thank the University Grants Commission (UGC), India for the financial support (UGC-Ref. No.: 1336/(CSIR-UGC NET JUNE 2019)).

%%%%%%%%%%%%%%%%%%%%%%%%%%%%%%%%%%%%%%%%%%%%%%%%%%%

\end{document}